\documentclass[12pt]{amsart}

\theoremstyle{plain}
\newtheorem{thm}{Theorem}[section]
\newtheorem{theorem}[thm]{Theorem}

\newtheorem{lemma}[thm]{Lemma}
\newtheorem{corollary}[thm]{Corollary}
\newtheorem{proposition}[thm]{Proposition}
\theoremstyle{definition}
\newtheorem{remark}[thm]{Remark}

\newtheorem{notation}[thm]{Notation}
\newtheorem{definition}[thm]{Definition}

\newtheorem{question}[thm]{Question}

\numberwithin{equation}{section}

\newcommand{\sB}{{\mathcal B}}

\newcommand{\sH}{{\mathcal H}}

\newcommand{\BO}{{\mathbb O}}

\newcommand{\C}{{\mathbb C}}

\newcommand{\BP}{{\mathbb P}}

\newcommand{\BS}{{\mathbb S}}

\newcommand{\End}{{\rm End}}

\newcommand{\fsl}{{\mathfrak s}{\mathfrak l}}
\newcommand{\fgl}{{\mathfrak g}{\mathfrak l}}

\newcommand{\aut}{{\mathfrak a}{\mathfrak u}{\mathfrak t}}

\def\Sym{\mathop{\rm Sym}\nolimits}

\def\Hom{\mathop{\rm Hom}\nolimits}

\title[Automorphisms of cubic hypersurfaces]{Prolongations of infinitesimal automorphisms of cubic hypersurfaces with nonzero Hessian}
\author[Jun-Muk Hwang]{Jun-Muk Hwang} 
\address{Korea Institute for Advanced Study, Hoegiro 87, Seoul 02455, Korea}
\email{jmhwang@kias.re.kr}
\thanks{The author is supported
by National Researcher Program 2010-0020413 of NRF}

\begin{document}

\begin{abstract}
We study the connected component of the automorphism group of a cubic hypersurface over complex numbers. When the cubic hypersurface has nonzero Hessian,  this group is usually small. But there are examples with unusually large automorphism groups: the secants of Severi varieties.
Can we characterize them by the property of having unusually large automorphism groups?
 We study this question from the viewpoint of  prolongations of the Lie algebras.
   Our result characterizes the secants of Severi varieties, among cubic hypersurfaces with nonzero Hessian and smooth singular locus, in terms of prolongations of certain type.
\end{abstract}

\maketitle



\section{Introduction}
Throughout, we will work over $\C$.
Let
$Y \subset \BP W$ be an irreducible reduced cubic hypersurface defined by a cubic form $f \in \Sym^3 W^*$ on a vector space $W$. We can consider the following three nondegeneracy conditions on the cubic form.

 \begin{itemize}
  \item[(ND1)]
  For any $w \in W$, there exists $u, v \in W$ such that $f(w, u, v) \neq 0.$
  \item[(ND2)]
 There exists $w \in W$ such that if $u \neq 0,$  then $f(w, u, v) \neq 0$ for some $v \in W$. \item[(ND3)] For any $w \in W$, there exists $u \in W$ such that $f(w, w, u) \neq 0$. \end{itemize}

There are obvious implications (ND3) $\Rightarrow$ (ND2) $\Rightarrow$ (ND1).
It is clear that $f$ satisfies (ND1) iff $Y$ is not a cone, and satisfies (ND3) iff $Y$ is nonsingular.
If $f$ satisfies (ND2), we say that $Y$ has nonzero Hessian (Definition \ref{d.Hessian}).
The main interest of this article is on the Lie algebra $\aut(Y) \subset \fsl(W)$ of infinitesimal linear automorphisms  of a cubic hypersurface $Y$ with nonzero Hessian.

If $f$ does not satisfy (ND1), namely, if $Y$ is a cone, then $\dim \aut(Y) \geq 1$.
On the other hand, if $f$ satisfies (ND3), then $\aut(Y) =0$. It seems reasonable to expect that
$\aut(Y)$ is not too big if $f$ satisfies (ND2).
 But there are cubics with nonzero Hessian that have many nontrivial infinitesimal automorphisms. The most famous ones are secants of Severi varieties (see  Theorem IV.4.7 in \cite{Za}), or equivalently,  homaloidal EKP cubics (see Theorem 3 in \cite{Do}). For the secant $Y = {\rm Sec}(S)$ of a Severi variety $S \subset \BP W$,  the Lie algebra   $\aut(Y)$ is equal to $\aut(S)$, which is semisimple and $S$ is homogeneous under $\aut(Y)$. This leads to the following question.

\begin{question}\label{q.vague}
  Let $Y \subset \BP W$ be an irreducible cubic hypersurface with nonzero Hessian. If $\aut(Y)$ is {\em unusually large}, is it the secant of a Severi variety? \end{question}

\medskip
Here we used the vague term `unusually large' intentionally so that the question can be interpreted in many different ways.
In this article, we will interpret it in terms of the prolongation  $$\aut(\widehat{Y})^{(1)}
\subset \Hom (S^2 W, W)$$
of the Lie algebra $$\aut(\widehat{Y}) = \C \cdot {\rm Id}_W + \aut(Y) \ \subset \ \fgl(W) = {\rm End}(W)$$ defined as follows.

\begin{definition}\label{d.prolong}
Let $W$ be a vector space.
For an element $A \in  \Hom (S^2 W, W),$ denote by $A_{uv} = A_{vu} \in W$ its value
at $u, v \in W$. For a projective variety $Z \subset \BP W$, we say that $A$ is a {\em prolongation} of $\aut(\widehat{Z})$ if
for each $w \in W$, the endomorphism $A_w \in {\rm End}(W)$ defined by $A_w (u) := A_{wu}$ belongs to $\aut(\widehat{Z})$. The vector space of all prolongations of $\aut(\widehat{Z})$
will be denoted by $\aut(\widehat{Z})^{(1)}.$ \end{definition}

The elements of $\aut(\widehat{Y})^{(1)}$  can be regarded as higher-order infinitesimal automorphisms of $Y$.  It is known that $\aut(\widehat{Y})^{(1)} \cong W^*$ when $Y$ is the secant of a Severi variety (e.g. Proposition 3.4 of \cite{FH12}). Thus  $\aut(\widehat{Y})^{(1)} \neq 0$ is one possible interpretation of   the phrase that $\aut(Y)$ is unusually large and we can refine Question \ref{q.vague} as follows.

\begin{question}\label{q} Let $Y$ be an irreducible cubic hypersurface with nonzero Hessian.
If $\aut(\widehat{Y})^{(1)} \neq 0,$ is $Y$ the secant of a Severi variety? \end{question}

For a cubic hypersurface $Y$, certain elements of  $\aut(\widehat{Y})^{(1)}$  are of particular interest. To define them, we need some notation first.
For $f \in \Sym^3 W^*$ defining $Y$ and two vectors $u, v \in W$, let $f_{uv} \in W^*$ be the linear functional $w \mapsto f_{uv}(w) = f( u, v, w)$. Each $A \in \aut(\widehat{Y})^{(1)}$ determines a linear functional $\chi^A \in W^*$, as explained in Lemma \ref{l.Aut}.

\begin{definition}\label{d.Xi}
Let $Y \subset \BP W$ be an irreducible cubic hypersurface defined by $f \in \Sym^3 W^*$, which has nonzero Hessian.
For a complex number $a \in \C$, let $\Xi^a_Y \subset \aut(\widehat{Y})^{(1)}$ be the subspace consisting of  $ A \in \aut(\widehat{Y})^{(1)}$ that satisfies $$ A_{uv}
=  a \chi^A(u) v + a \chi^A(v) u + h^A(f_{uv}) $$  for some $ h^A \in \Hom(W^*, W)$ and all $u, v \in W$.
 We will see in  Proposition \ref{p.unique} that the homomorphism $h^A$ is uniquely determined by $A$.\end{definition}

When $Y$ is the secant variety of a Severi variety, it can be shown that $\aut(\widehat{Y})^{(1)} = \Xi_Y^{1/2}$,
which  suggests the following weaker version of Question \ref{q}.

\begin{question}\label{q.Xi} Let $Y$ be an irreducible cubic hypersurface with nonzero Hessian.
If $\Xi_Y^{a} \neq 0$ for some $a \in \C$, is $Y$ the secant of a Severi variety?
If $\Xi_Y^{1/2} \neq 0$, is $Y$ the secant of a Severi variety? \end{question}

 Our main result is the following partial answer to Question \ref{q.Xi}.  Let ${\rm Sing}(Y)$ be the set-theoretic singular locus of $Y$.

\begin{theorem}\label{t.cubic} Let $Y \subset \BP W$ be an irreducible cubic hypersurface  with nonzero Hessian.
 Assume that \begin{itemize} \item[(a)] the singular locus
${\rm Sing}(Y)$ is nonsingular; and \item[(b)]
$\Xi_Y^a \neq 0$ for some $a \neq \frac{1}{4}.$ \end{itemize}
Then $Y$ is the secant variety of a Severi variety.
\end{theorem}

We would like to emphasize that ${\rm Sing}(Y)$ in  (a) is the reduced
algebraic set consisting of singular points of $Y$. In particular, the natural scheme structure of the singular locus of $Y$
is not assumed to be nonsingular. If we replace it by the stronger requirement that the scheme structure is nonsingular, then the proof becomes much simpler, but the result will be less useful.

  The definition of $\Xi_Y^a$ and the condition $a \neq \frac{1}{4}$ in (b) may look technical.
  But  this  condition is checkable practically. In many concrete problems, when an element $A \in \aut(\widehat{Z})^{(1)}$ for a projective variety $Z \subset \BP W$ appears, we often have such additional information.

The proof of Theorem \ref{t.cubic} is done in two steps. The first step is to show that $Y$ is equal to
${\rm Sec}({\rm Sing}(Y))$, the secant variety of the singular locus of $Y$ (Theorem \ref{t.cubic1}). In this step, the condition $a \neq \frac{1}{4}$ in (b) is used crucially, while the assumption (a) is not used. This step involves a detailed study of the Gauss map of $Y$ and its relation to $\aut(\widehat{Y})^{(1)}.$
The second step
is to show under the assumption $Y = {\rm Sec}({\rm Sing}(Y))$,  that   there exists an irreducible component $S$ of ${\rm Sing}(Y)$ such that $Y= {\rm Sec}(S)$ (Theorem \ref{t.cubic2}). The second step does use the assumption (a). Once these two steps are established, Theorem \ref{t.cubic} follows easily from the classification results of \cite{FH12} and \cite{FH16}, which will be reviewed in Section \ref{s.FH}.

\begin{notation}\label{n} \begin{itemize}
\item[1]. For an algebraic variety $Z$, the singular locus ${\rm Sing}(Z)$ is the (reduced) algebraic subset consisting of singular points of $Z$ and the smooth locus ${\rm Sm}(Z)$ is its complement.
For a nonsingular point $x \in {\rm Sm}(Z)$, the (abstract) tangent space of $Z$ at $x$ is denoted by $T_x(Z)$.
\item[2]. Let $W$ be a vector space. For a projective variety $Z \subset \BP W$, its affine cone is denoted by $\widehat{Z} \subset W$. For a point $x \in \BP W$, the corresponding 1-dimensional subspace of $W$ is $\widehat{x} \subset W$. For a nonzero vector $w\in W$, the corresponding point of $\BP W$ is $[w] \in \BP W$ so that  $\C w = \widehat{[w]}.$  For a nonsingular point $w \in \widehat{Z}$, the affine tangent space of $\widehat{Z}$ at $w$ will be denoted by $T_w\widehat{Z}  \subset W.$
When $x = [w] \in \BP W$, we often write  by abuse of notation $\widehat{w}$ in place of $\widehat{x}$ and
$T_x\widehat{Z}$  in place of $T_w\widehat{Z}.$
\item[3]. For a projective variety $Z \subset \BP W$, the Lie algebra of infinitesimal automorphisms of $\widehat{Z}$ is
$$\aut(\widehat{Z}) := \{ \varphi \in {\rm End}(W), \ \varphi(w) \in T_w\widehat{Z} \mbox{ for each } w \in  {\rm Sm}(\widehat{Z}) \}.$$ The Lie algebra of infinitesimal automorphisms of $Z$ is $\aut(Z) = \aut(\widehat{Z}) \cap \fsl(W).$
\item[4]. For an algebraic subset $Z \subset \BP W$, the secant variety ${\rm Sec}(Z)$ is the closure of the union of the lines $\overline{xy}$ joining any pair of distinct points $x \neq y \in Z$. \end{itemize}
\end{notation}

\section{Characterizing Severi varieties in terms of prolongations}\label{s.FH}

This section is of auxiliary nature, somewhat independent from the rest of the paper.
 Our goal is to prove the following characterization of Severi varieties, which will be used in Section \ref{s.cubic2}.

\begin{theorem}\label{t.Severi}
 Let $S \subset \BP W$ be an irreducible nondegenerate nonsingular   subvariety with $\aut(\widehat{S})^{(1)} \neq 0.$  If the secant variety ${\rm Sec}(S) \subset \BP W$ is a
hypersurface, then $S \subset \BP V$ is one of the following four Severi varieties (all of them belonging to (i) of
Theorem \ref{t.FH12} below):
 $$v_2(\BP^2) \subset \BP^5, \ \BP^2 \times \BP^2 \subset \BP^8, \
{\rm Gr}(2, 6) \subset \BP^{14}, \
  \mbox{ and } \
\BO \BP^2 \subset \BP^{26}. $$
 \end{theorem}

The proof of Theorem \ref{t.Severi} is a straight-forward application of the classification of
irreducible nondegenerate nonsingular  subvarieties  $S \subset \BP W$ with $\aut(\widehat{S})^{(1)} \neq 0$, obtained in \cite{FH12} and \cite{FH16} (an error in \cite{FH12} has been corrected in \cite{FH16}, adding three more varieties that were overlooked in \cite{FH12}). The classification can be summarized as follows.

\begin{theorem}\label{t.FH12}[Theorem 7.13 of \cite{FH16}]
 Let $S \subset \BP W$ be an irreducible nondegenerate nonsingular subvariety with $\aut(\widehat{S})^{(1)} \neq 0.$ Then $S$ is projectively equivalent to one of the followings.
   \begin{itemize} \item[(i)] VMRT of an irreducible Hermitian symmetric space (explained in Example 4.4 of \cite{FH16})
   \item[(ii)] VMRT of a (both even and odd) symplectic Grassmannian (explained in Example 4.5 of \cite{FH16})
   \item[(iii)] a nonsingular linear section of ${\rm Gr}(2, \C^5) \subset \BP^9$ of dimension 4 or 5
       \item[(iv)] some nonsingular linear section of the Spinor variety $\BS_{5} \subset \BP^{15}$ of dimension 7, 8 or 9
           \item[(v)] certain biregular projections of varieties in (i) or (ii) (
           described in Section 4 of \cite{FH12}). \end{itemize}
               \end{theorem}

To prove Theorem \ref{t.Severi}, it is convenient to introduce the following definition, a variation of Definition 4.20 of \cite{FH12}.

            \begin{definition}\label{d.submaximal}
            Let $Z \subset \BP V$ be a nondegenerate variety. A linear subspace
            $\BP L \subset \BP V \setminus {\rm Sec}(Z)$ is {\em submaximal} if under the projection $p_L: Z \to \BP (V/L)$, the secant variety of $p_L(Z)$ is a hypersurface in $\BP (V/L)$. In this case, $\dim L = \dim \BP V - \dim {\rm Sec}(Z) -1$. \end{definition}

            The following proposition is a slight variation of Theorem 4.21 of \cite{FH12}. The only difference is that `maximal' there is replaced by `submaximal' here.

            \begin{proposition}\label{p.submaximal}
            Let $Z \subset \BP V$ be one of the varieties in (i) and (ii) in Theorem \ref{t.FH12}  with $\dim {\rm Sec}(Z) \leq \dim \BP V -2$. Let $\BP L \subset \BP V \setminus {\rm Sec}(Z)$ be a submaximal linear subspace. Then $\aut(\widehat{p_L(Z)})^{(1)} = 0. $
            \end{proposition}

 For convenience,   we will prove  Theorem \ref{t.Severi} first, assuming  Proposition \ref{p.submaximal} whose proof will be given afterwards.

\begin{proof}[Proof of Theorem \ref{t.Severi}]
We will go through the list of varieties in Theorem \ref{t.FH12}, checking when
${\rm Sec}(S) \subset \BP W$ is a hypersurface.

 Note that the secant variety of a member in (iii) or (iv) in Theorem \ref{t.FH12} is equal to the ambient projective space. This can be seen, for example, by Zak's theorem on linear normality (see \cite{Za} Corollary II.2.11) : a nondegenerate nonsingular variety $S \subset \BP^m$ of dimension
$n$ satisfying $3n > 2(m-2)$ has ${\rm Sec}(S) = \BP^m$.

When $S$ is one of (i) in Theorem \ref{t.FH12}, we can list the dimension of $S$ and ${\rm Sec}(S)$ as follows from
Table p.446 of \cite{FH12}. (The roman numerals indicate the types of the Hermitian symmetric spaces.)

\begin{itemize}
\item[(I)] Segre $S= \BP^{a-1} \times \BP^{b-1} \subset \BP^{ab-1}$ and $\dim {\rm Sec}(S) = 2a+ 2b-5$
    \item[(II)] Pl\"ucker $S= {\rm Gr}(2, n) \subset \BP^{\frac{n^2-n-2}{2}}$ and $\dim {\rm Sec}(S) = 4n-11$
        \item[(III)] 2nd Veronese $S = v_2(\BP^{n-1}) \subset \BP^{\frac{n^2 + n -2}{2}}$ and $\dim {\rm Sec}(S) = 2n-2$
            \item[(IV)] $S= Q^{n-2} \subset \BP^{n-1}$ and $\dim {\rm Sec}(S) = n-1$
            \item[(V)] $S= \BS_5 \subset \BP^{15}$ and $\dim {\rm Sec}(S) = 15$
            \item[(VI)] $S = \BO \BP^2 \subset \BP^{26}$ and $\dim {\rm Sec}(S) = 25$
            \end{itemize}
            From the above list, the only cases when ${\rm Sec}(S)$ is a hypersurface are
            exactly the four Severi varieties.

            When $S$ is one of (ii) in Theorem \ref{t.FH12},   there are positive integers $k$ and $m$ (see Lemma 4.19 of \cite{FH12}) such that $$S \subset \BP^{\frac{k^2+k}{2} + mk -1} \mbox{  and }  \dim {\rm Sec}(S) = 2m+2k -2.$$ So ${\rm Sec}(S)$ cannot be a hypersurface.

       Finally, when  $S \subset \BP W$ is one of (v) in Theorem \ref{t.FH12},
       Proposition \ref{p.submaximal} says that ${\rm Sec}(S) \subset \BP W$ is not a hypersurface. \end{proof}

                   \begin{proof}[Proof of Proposition \ref{p.submaximal}]
                     The proof is essentially the same as the proof of Theorem 4.21 of \cite{FH12}.
            The dimension $$\dim L = \dim \BP V - \dim {\rm Sec}(S) -1$$ of a submaximal $L \subset V$  in each of the cases of (i) and (ii)
              can be computed from the dimension information listed in the proof of
               Theorem \ref{t.Severi}. We list them below.
\begin{itemize} \item[(I)]   $\dim L =   ab - 2a-2b +3$  \item[(II)] $\dim L =  \frac{(n^2-n-2)}{2} - 4n+10$
\item[(III)] $\dim L = \frac{(n^2 + n  -2)}{2} - 2n +1$ \item[(ii)] $\dim L = mk + \frac{k(k+1)}{2} -2m -2k$\end{itemize}
            For the varieties in (i) and  (ii),
            an upper bound on $\dim L$ under the condition that $\aut(\widehat{p_L(Z)})^{(1)} \neq 0$ is given in Propositions 4.10 (iii), 4.11 (iii), 4.12 (iii) and 4.18 (iii) of \cite{FH12}. We cite them below, where  the rank $r$ of an element of $\aut(\widehat{p_L(Z)})^{(1)}$ must be at least 1 if  $\aut(\widehat{p_L(Z)})^{(1)}$ is nonzero.
            \begin{itemize}
            \item[(I)] $\dim L \leq ab -2(a+b) + 4 - r(a +b -r-4)$
            \item[(II)] $\dim L \leq \frac{n(n-1)}{2} - 4n + 10 - \frac{r(2n-r-9)}{2}$
            \item[(III)] $\dim L \leq \frac{n(n+1)}{2} -2n+1 - \frac{r(2n-r-3)}{2}$
            \item[(ii)] $\dim L \leq mk + \frac{k(k+1)}{2} - 2m-2k +1 - \frac{r(2m+2k-r-3)}{2}$ \end{itemize}
                If $r \geq 1$, these dimensions are strictly smaller than the dimensions
                in the previous list for the submaximal cases.
                       Thus  $\aut(\widehat{p_L(Z)})^{(1)}$ must vanish when $\BP L$ is submaximal.
                \end{proof}

\section{Basic results on cubic hypersurfaces}\label{s.basic}
In this section, we present  some basic definitions and results on cubic hypersurfaces. Throughout, we fix an irreducible nonzero cubic form $f\in \Sym^3 W^*$ on a vector space $W$:
$$ f(u, v, w) = f(u, w, v) = f(w, u, v) \mbox{ for all } u, v, w \in W,$$ and denote by  $Y  \subset \BP W$ the cubic hypersurface whose affine cone is   $$\widehat{Y} = \{ w \in W, \ f(w, w, w) =0\}.$$

\begin{definition}\label{d.cubic}
 For $v, w \in W$, define $f_w \in \Sym^2 W^*$  and $f_{vw} \in W^*$ respectively by $$f_w (u, s) := f(w, u, s) \mbox{ for all } u, s \in W $$ $$f_{vw}(u) := f (v, w, u) \mbox{ for all } u \in W.$$ The {\em null space} of the quadratic form $f_w$ is
\begin{eqnarray*} {\rm Null}( f_w) & :=  &\{ v \in W, \ f_{wv} =0\} \\  &=& \{ v \in W, \ f(w,v,u) =0 \mbox{ for all } u \in W\}. \end{eqnarray*}  The {\em Gauss space}
of $w \in W$ is defined by $$ \Gamma_w := \{ v \in W, f_{vw} \in \C \cdot f_{ww} \subset W^* \}.$$ When $w \neq 0$ and $x= [w] \in \BP W$, we will sometimes write $\Gamma_w$ as $\Gamma_x$ by abuse of notation.\end{definition}

  We will skip the proof of the following elementary lemma.

   \begin{lemma}\label{l.Sing}
 The singular locus ${\rm Sing}(Y) \subset Y$ is  the algebraic subset defined by the affine cone $$ {\rm Sing}(\widehat{Y})=  \widehat{{\rm Sing}(Y)} = \{ w \in W, \ f_{ww} = 0 \} $$ and its  secant variety ${\rm Sec}({\rm Sing}(Y))$ is contained in $Y$.  For $w \in {\rm Sm}(\widehat{Y}),$ the affine tangent space is given by $$
T_w\widehat{Y} =  \{ v \in W, f_{ww}(v) =0\}.$$ \end{lemma}

  \begin{definition}\label{d.polar}
  The  {\em polar map} of $f$ (or $Y$) is the rational map $\Phi: \BP W \dasharrow \BP W^*$ defined for each $w \in  W \setminus  \widehat{  {\rm Sing}(Y)}$ by     $ \Phi ([w])  := [f_{ww}].$ From Lemma \ref{l.Sing}, the base locus of $\Phi$ is exactly ${\rm Sing}(Y).$ \end{definition}

\begin{lemma}\label{l.Phi}
In Definition \ref{d.cubic} and Definition \ref{d.polar}, we have the following.  \begin{itemize} \item[(1)]   For any $w \in W$, the Gauss space $\Gamma_w$ is spanned by  $ \widehat{w}$ and $ {\rm Null}(f_w)$. If $w
 \in {\rm Sm}(\widehat{Y})$, then $\Gamma_w \subset T_w\widehat{Y}$.
 \item[(2)] If $w \not\in \widehat{{\rm Sing}(Y)}$, then $w \not\in {\rm Null}(f_w)$.
\item[(3)] For $x \not\in {\rm Sing}(Y),$
consider   the derivative ${\rm d}_x \Phi: T_x(\BP W) \to T_{\Phi(x)}(\BP W^*)$. In terms of the natural identification $T_x(\BP W) = \Hom (\widehat{x}, W/\widehat{x})$, we have
$${\rm Ker} ({\rm d}_x \Phi) = \Hom (\widehat{x}, \Gamma_x/\widehat{x}).$$
\item[(4)] The polar map $\Phi$ is dominant  if and only if ${\rm Null}(f_w) =0$ for a general $w \in W$. \end{itemize} \end{lemma}

\begin{proof}
For any $v \in \Gamma_w$, we have $f_{vw} = c \ f_{ww}$ for some $c \in \C$. Thus $f(w, cw-v, u) = 0$ for all $u \in W$.
It follows that $cw -v \in {\rm Null}(f_w)$. Conversely, if $v \in {\rm Null}(f_w)$, then $$f_{w, w+v} = f_{ww} + f_{wv}  = f_{ww}.$$
Thus $w+v \in \Gamma_w.$  When $x$ is a nonsingular point of $\widehat{Y}$, it is obvious from $$T_w(\widehat{Y}) = \{ v \in W, \ f(w,w, v) =0\}$$ that $\widehat{w}$ and ${\rm Null}(f_w)$ are contained in $T_w(\widehat{Y})$. This proves (1).

If $w \in {\rm Null}(f_w)$, then $f_{ww} =0$. Thus $[w] \in {\rm Sing}(Y)$, proving (2).

To prove (3), pick $ v \in W$ and let $\vec{v} \in \Hom(\widehat{x}, W/\widehat{x})$ be the homomorphism sending some nonzero $w \in \widehat{x}$ to $v$.
Then for $t \in \C$ close to 0,  $$f_{w+tv, w+tv} = f_{ww} + 2 t f_{wv} + t^2 f_{vv}$$ shows that ${\rm d}_x \Phi(\vec{v})$ is the element of $$T_{\Phi(x)} = \Hom( \widehat{f_{ww}}, W^*/\widehat{f_{ww}})$$ sending $f_{ww}$ to $2 f_{wv}$. It follows that $\vec{v} \in {\rm Ker} ({\rm d}_x \Phi)$ if and only if $f_{wv} \in \widehat{f_{ww}}$, which is equivalent to saying $v \in \Gamma_x$.

For (4), note that $\Phi$ is dominant if and only if ${\rm Ker} ({\rm d}_x \Phi)=0$ for a general $x \in \BP W$.
By (3), this is equivalent to saying that $\Gamma_x = \widehat{x}$ for a general $x$. By (1) and (2), this is equivalent to
${\rm Null}(f_w) =0$ for a general $w \in W$. \end{proof}

\begin{definition}\label{d.Hessian}
In the setting of Lemma \ref{l.Phi},
we say that $Y$ (or $f$) {\em has nonzero Hessian}, or equivalently, is a {\em cubic with nonzero Hessian}, if the two equivalent conditions in Lemma \ref{l.Phi} (4) hold. \end{definition}

We will skip the proof of the following elementary lemma.

\begin{lemma}\label{l.Aut} In Definition \ref{d.cubic}, we have the following.
\begin{itemize} \item[(1)] The Lie algebra $\aut(\widehat{Y}) $ of infinitesimal automorphisms of $\widehat{Y} \subset W$
       consists of endomorphisms $ \varphi \in {\rm End}(W)$ satisfying $$ f(\varphi (u), v, w) + f(u, \varphi (v), w) + f(u, v, \varphi(w)) = \chi(\varphi) \ f(u, v, w) $$ for all $u,v,w \in W$ and for some $ \chi(\varphi) \in \C.$
Note that $\chi$ defines a linear functional $\chi: \aut(\widehat{Y}) \to \C$.

\item[(2)] For an element $A \in  \Hom (S^2 W, W)$, denote by $A_w \in {\rm End}(W)$ the endomorphism defined by $A_w (u) := A_{wu}$ and  denote by $A_{wv} = A_{vw} \in W$ its value at $v \in W$.
For each $A \in \aut(\widehat{Y})^{(1)}$, denote by $\chi^A \in W^* $ the linear functional defined by $\chi^A (w) := \chi (A_w). $
Then $A \in \aut(\widehat{Y})^{(1)}$ satisfies
$$f(A_{wu}, v, s) + f(u, A_{wv}, s) + f( u, v, A_{ws}) = \chi^A(w) f (u, v, s)$$
for each $s, u, v, w \in W$.
\end{itemize} \end{lemma}

\begin{proposition}\label{p.injective}
 If $Y$ has nonzero Hessian, then the linear homomorphism $\aut(\widehat{Y})^{(1)} \to W^*$   given by $A \mapsto \chi^A$ is injective.
 \end{proposition}

 \begin{proof}
Assume that $\chi^A = 0$. By Lemma \ref{l.Aut} (2), for all $w, u \in W$,
$$f(A_{ww}, w, u) + f(w, A_{ww}, u) + f(w,w, A_{wu}) = 0,$$ which implies $$2 \ f(A_{ww}, w, u) = - f (w, w, A_{wu}).$$
On the other hand, $$0 = f(A_{uw}, w, w) + f(w, A_{uw}, w) + f(w, w, A_{uw}) = 3 f(w, w, A_{wu}).$$ It follows that $f( A_{ww}, w, u) =0$
for all  $u \in W$, i.e., $$A_{ww} \in {\rm Null} (f_w).$$ But ${\rm Null}(f_w) =0$ for a general $w \in W$ because $Y$ has nonzero Hessian. Thus $A_{ww} = 0$ for a general $w$. It follows that $A_{ww} = 0$ for all $w \in W$.
\end{proof}

\begin{proposition}\label{p.Gamma}
Let $Y \subset \BP W$ be as in Lemma \ref{l.Aut}.
  For any $A \in \aut(\widehat{Y})^{(1)}$ and any $w \in \widehat{Y}$, we have $A_{ww} \in \Gamma_w.$ More precisely,
    $$ f(A_{ww}, w, u) = \frac{\chi^A(w)}{2} f(w,w, u) \mbox{ for all } u \in W,$$  which implies $f_{A_{ww}, w} \in \C \cdot f_{ww}$.
\end{proposition}

\begin{proof}
By Lemma \ref{l.Aut} (2),
 $$ f(A_{uw}, w, w) + f(w, A_{uw}, w) + f(w, w, A_{uw}) = \chi^A(u) f(w, w, w) = 0,$$
 which gives  $f(w,w, A_{wu}) = 0.$
Putting it in
$$f(A_{ww}, w, u) + f(w, A_{ww}, u) + f(w, w, A_{wu}) = \chi^A(w) f(w, w, u), $$ we obtain $$ f(A_{ww}, w, u) = \frac{\chi^A(w)}{2} f(w,w, u).$$
 \end{proof}

\section{Prolongations and Gauss map of a cubic hypersurface}\label{s.II}

\begin{definition}\label{d.Gauss}
For an irreducible cubic hypersurface $Y \subset \BP W$, its {\em affine Gauss map}
$$\gamma_{\widehat{Y}}: \widehat{Y} \dasharrow W^*$$
associates to $w \in {\rm Sm}(\widehat{Y})$ its affine tangent space $T_w\widehat{Y} \subset W$. It induces the {\em Gauss map}
 $\gamma_Y: Y \dasharrow \BP W^*.$
\end{definition}

We have the following standard result.

\begin{proposition}\label{p.Gauss}
For an irreducible cubic hypersurface with nonzero Hessian,  we have the following.  \begin{itemize}
\item[(1)] The restriction of the polar map $\Phi|_Y: Y \dasharrow \BP W^*$ is the Gauss map $\gamma_Y$ of $Y$ and its proper image is the
dual variety $Y^* \subset \BP W^*$ of $Y$.
\item[(2)]
The fiber of $\gamma_{\widehat{Y}}$ through a general point $w \in \widehat{Y}$ coincides with $\Gamma_w$.
\item[(3)] Let $F \subset Y$ be a linear subspace contained in $Y$ such that
 $F \not\subset {\rm Sing}(Y)$ and $\gamma_Y(F \setminus {\rm Sing}(Y))$ is a single point. Then the set-theoretic intersection $F \cap {\rm Sing}(Y)$ is a hypersurface of degree $\leq 2$ in the linear space $F$.
     \end{itemize} \end{proposition}

\begin{proof}
(1) is a direct consequence of $T_w \widehat{Y} = \{ v \in W, \ f_{ww} (v) = 0\}$ for $w \in {\rm Sm}(\widehat{Y})$.

(2) is a special case of a more general result, Theorem 2.4.6 in \cite{FP}.
In our situation, it is an immediate consequence of the classical fact that a general fiber of the Gauss map is a linear subspace (Linearity Theorem in Section 2.3.2 of \cite{FP}).
In fact, the assumption that $\Phi$ is dominant implies that the fiber of $\gamma_{Y}$ through a general point $[w] \in Y$ is equal to the fiber of $\Phi$ through $[w]$. By Lemma \ref{l.Phi} (3), the tangent to this fiber at $[w]$ corresponds to $\Gamma_w$. Thus the fiber of $\gamma_Y$ must be $\BP \Gamma_w$.

To prove (3),
fix a general $w \in \widehat{F}.$ We claim that $f(v,v, u) =0$ for any $v \in \widehat{F}$ and $u \in T_w\widehat{Y}$. It is sufficient to prove this claim for a general $v \in \widehat{F}$.
But $T_v\widehat{Y} = T_w\widehat{Y}$ by the assumption that $\gamma_Y([v]) = \gamma_Y([w])$.
Thus $f(v,v, u) =0$ from $u \in T_v \widehat{Y}= T_w \widehat{Y},$ proving the claim.

Now, fix a vector $u' \in W \setminus T_w\widehat{Y}$.
 Recall
$$\widehat{F} \cap {\rm Sing}(\widehat{Y}) = \{ v \in \widehat{F}, \ f(v,v,u) =0 \mbox{ for all }
u \in W \}.$$
By the above claim, this is equal to $$\{ v \in \widehat{F}, \ f(v, v, u') =0 \}.$$
This is a hypersurface in $\widehat{F}$ defined by a quadratic equation, from which (3) follows. \end{proof}

\begin{proposition}\label{p.varphi}
Let $Y \subset \BP W$ be as in Proposition \ref{p.Gauss} and pick a general point $w \in {\rm Sm}(\widehat{Y}).$
\begin{itemize} \item[(1)] If $\varphi \in {\rm aut}(\widehat{Y})$ satisfies
$\varphi(w) \in \Gamma_w$, then $\varphi(T_w \widehat{Y}) \subset T_w \widehat{Y}$.
\item[(2)] If $\varphi \in {\rm aut}(\widehat{Y})$ satisfies
$\varphi(w) \in \Gamma_w$, then $\varphi(\Gamma_w) \subset \Gamma_w$.
\item[(3)] If $A \in {\rm aut}(\widehat{Y})^{(1)}$, then $A_w (\Gamma_w) \subset \Gamma_w.$
\end{itemize}
\end{proposition}

\begin{proof}
From Lemma \ref{l.Aut} (1), we have for any $u \in T_w \widehat{Y}$,
$$f(\varphi(w), w, u) + f(w, \varphi(w), u) + f(w, w, \varphi(u)) = \chi(\varphi) \ f(w, w, u) =0.$$ By the assumption $\varphi(w) \in \Gamma_w$, we have $$f(\varphi(w), w, u) = f(w, \varphi(w), u) = 0.$$ Thus $f(w, w, \varphi(u)) =0$, which means $\varphi(u) \in T_w \widehat{Y}$. This verifies (1).

To prove (2), we may assume that $\dim \Gamma_w \geq 2$.
Pick a general
$v \in \Gamma_w$ such that $\Gamma_w = \Gamma_v$ and $T_w \widehat{Y} = T_v \widehat{Y}$ by Proposition \ref{p.Gauss} (2). The assumption $\varphi(w) \in \Gamma_w =  \Gamma_v$ implies
\begin{equation}\label{e.varphi} f(v, \varphi(w), u) = 0 \mbox{ for any }u \in T_v \widehat{Y} = T_w \widehat{Y}.\end{equation}
From Lemma \ref{l.Aut}, we have  for any $u \in T_w \widehat{Y}$, $$f(\varphi(v), w, u) + f(v, \varphi(w), u) + f(v, w, \varphi(u)) = \chi(\varphi) f(v, w, u) =0.$$ On the left hand side,
the second term vanishes by (\ref{e.varphi}) and the third term vanishes from (1).
Thus we have $f(\varphi(v), w, u) =0$, proving (2).

(3) is immediate by applying (2) to $\varphi:= A_w$ which satisfies $\varphi(w) \in \Gamma_w$
by Proposition \ref{p.Gamma}. \end{proof}

     \begin{proposition}\label{p.A_xx}
    Let $Y \subset \BP W$ be an irreducible cubic hypersurface. For any nonzero  $A \in \aut({\widehat{Y}})^{(1)}$ and a general $w \in \widehat{Y}$, the vector $A_{ww}$ is not contained in $\widehat{w}$.  \end{proposition}

    \begin{proof}
    Assuming that  $A_{ww} \in \widehat{w}$ for all $w \in \widehat{Y}$,   let us draw a contradiction.
    Using the linear normality of $Y \subset \BP W$, we have
    a linear functional $\lambda \in W^*$ such that $$A_{ww} = \lambda(w) w   \mbox{ for all } w \in \widehat{Y}.$$ Thus  $$ w \mapsto A_{ww} - \lambda(w) w $$ determines a system of quadratic equations satisfied by $Y$. But  $Y$ is an irreducible cubic hypersurface. Thus we have $A_{ww} = \lambda(w)w$ for all $w \in W$ and \begin{equation}\label{e.lambda} A_{uv} = \frac{\lambda(u)}{2} v + \frac{\lambda(v)}{2} u \mbox{ for all } u, v \in W. \end{equation}
    From Lemma \ref{l.Aut} (2), we have
    $$f (A_{ww}, w, w) + f(w, A_{ww}, w) + f(w, w, A_{ww}) = \chi^A(w) \ f(w,w,w)$$ for any $w \in W$.  Since the left hand  side is $$ 3 \ f(A_{ww}, w, w) = 3\ f(\lambda(w) w, w, w) = 3 \lambda(w) \ f(w,w,w),$$
   we have $\chi^A = 3 \lambda.$Thus  $$ f( A_{wu}, u, u) + f(u, A_{wu}, u) + f(u, u, A_{wu}) = \chi^A(w)\ f(u, u, u)$$ gives
   $$  f(A_{wu}, u, u) =  \lambda(w) f(u, u, u) \mbox{ for all } u, w \in W.$$
     Since we have from (\ref{e.lambda}) $$ f(A_{wu}, u, u) = \frac{\lambda(w)}{2} f(u, u, u) + \frac{\lambda(u)}{2} f(w, u, u),$$ we obtain $$\frac{\lambda(u)}{2} f(w, u, u) = \frac{\lambda(w)}{2} f(u, u, u) \mbox{ for all } w, u \in W.$$ We know that $\lambda \in W^*$ is nonzero from (\ref{e.lambda}). So we can fix $w$ with $\lambda(w) \neq 0$ in the last equation to have  $f(u, u, u) = 0$ if $\lambda(u) =0$. This is impossible because $Y$ is an irreducible cubic hypersurface.
    \end{proof}

By Proposition \ref{p.Gamma} and Proposition \ref{p.A_xx}, if  $ {\rm aut}(\widehat{Y})^{(1)} \neq 0,$ then we have
$\dim \Gamma_w \geq 2$ for a general $w \in \widehat{Y}$. Combining it with Proposition \ref{p.Gauss} (2), we obtain

     \begin{corollary}\label{c.degdual}
     Let $Y \subset \BP W$ be an irreducible cubic hypersurface with nonzero Hessian. If $ {\rm aut}(\widehat{Y})^{(1)} \neq 0,$  then the Gauss map $\gamma_{Y}: Y \dasharrow \BP W^*$ is degenerate, i.e., the dual variety $Y^* \subset \BP W^*$ has codimension at least 2.
     \end{corollary}

\begin{remark}\label{r.Zak}
Corollary \ref{c.degdual} shows that an answer to Question \ref{q}, as well as a proof of
 Theorem \ref{t.cubic}, can be obtained if one can classify projective subvarieties of codimension at least two whose dual varieties are cubic hypersurfaces. Zak has classified nonsingular varieties satisfying these properties (Section IV.5 of \cite{Za}). But there are serious difficulties in extending his arguments to singular varieties.
 \end{remark}

The next proposition shows that in Definition \ref{d.Xi}, the homomorphism $h^A$ is uniquely determined by $A$.

\begin{proposition}\label{p.unique}
Let  $Y \subset \BP W$ be an irreducible cubic hypersurface with nonzero Hessian.
Assume that we have  $0 \neq A \in \aut(\widehat{Y})^{(1)}$ of the form
$$A_{ww} = 2a \chi^A(w) w + h (f_{ww}) \mbox{ for } w \in W$$ for some $h \in \Hom(W^*, W)$ and $a \in \C$. Then $a$ and $h$ are uniquely determined by $A$. \end{proposition}

\begin{proof}
Assume that we have $a, a' \in \C$ and $h, h' \in \Hom(W^*, W)$ such that
$$ A_{ww} = 2a \chi^A(w) w + h (f_{ww}) = 2 a' \chi^A(w) w + h'(f_{ww}) \mbox{ for } w \in W.$$
 If $a \neq a'$, then $(h-h')(f_{ww}) = 2 (a'-a) \chi^A(w) w$ implies that the constructible set  $$  J:= \{ (h-h')(f_{ww}), \ w \in \widehat{Y} \} $$ contains an open subset of $\widehat{Y}$. Proposition \ref{p.Gauss} (1) shows that $J$ is a subset of   $(h-h')(\widehat{Y}^*)$ for the dual variety $Y^* \subset \BP W^*.$ Thus Corollary \ref{c.degdual} implies that $\dim J \leq \dim \widehat{Y}^* < \dim \widehat{Y}$, a contradiction. We conclude that $a=a'$.

Now we have $(h-h') (f_{ww}) = 0 $ for all $w \in W$. Since $\Phi$ is dominant, we obtain $h=h'$. \end{proof}

 We will skip the straightforward proof of the following lemma.

\begin{lemma}\label{l.action}
\begin{itemize} \item[(1)] Let ${\rm Aut}_o(\widehat{Y}) \subset {\bf GL}(W)$ be the subgroup corresponding to the Lie algebra $\aut(\widehat{Y}).$ Then there exists a homomorphism $$e^{\chi}: {\rm Aut}_o(\widehat{Y}) \to \C^*$$ such that  every $g \in {\rm Aut}_o(\widehat{Y})$ satisfies  $$f(g\cdot u, g \cdot v, g \cdot w) =e^{\chi} (g)  \ f(u, v, w)$$ for all $u, v, w \in W.$
\item[(2)] The action of $g \in {\rm Aut}_o(\widehat{Y})$
on  $A \in \aut(\widehat{Y})^{(1)}$  defined by $$(g \cdot A)_{uv} = g \cdot (A_{g^{-1}u, g^{-1}v})$$ satisfies $$(g \cdot A)_u = g \cdot (A_{g^{-1}u}) \mbox{ and } \chi^{g \cdot A} = g \cdot \chi^A$$ where the action of $g$ in the righthand side of the first (resp. second) equality  refers to  the induced action of $g$ on ${\rm End}(W)$ (resp. $W^*$). \end{itemize} \end{lemma}

\begin{proposition}\label{p.invariance}
For each $a \in \C$, the subspace $\Xi^a_Y \subset \aut(\widehat{Y})^{(1)}$ of Definition \ref{d.Xi} is preserved under the action of ${\rm Aut}_o(\widehat{Y})$ on
$\aut(\widehat{Y})^{(1)}$ described in Lemma \ref{l.action}. \end{proposition}

\begin{proof}
For any $u, v, w \in W$, $h \in \Hom(W^*, W)$ and $g \in {\rm Aut}_o(\widehat{Y})$, we have \begin{eqnarray*}
f_{g^{-1}u, g^{-1}v}(w) & =& f( g^{-1} u, g^{-1} v, w) \\ &=& e^{\chi}(g^{-1}) \ f(u, v, g w) \\ &=& e^{\chi}(g^{-1}) \ f_{uv}(gw) \\ &=& e^{\chi}(g^{-1}) \ ( g^{-1} \cdot f_{uv}) (w). \end{eqnarray*}
Thus \begin{equation}\label{e1}  g \left( h (f_{g^{-1}u, f^{-1}v}) \right) = g \left( h \left( e^{\chi}(g^{-1}) g^{-1} \cdot f_{uv} \right) \right). \end{equation}
Given $A \in \Xi^a_Y$, set $$h^{g \cdot A} := g \circ h^A \circ \left( e^{\chi}(g^{-1}) g^{-1} \right) \in \Hom (W^*, W),$$ where $ e^{\chi}(g^{-1}) g^{-1}$ is viewed as an element of $\End(W^*).$ Then (\ref{e1}) implies
\begin{equation}\label{e2}  h^{g \cdot A}(f_{uu}) = g \cdot \left( h^A ( f_{g^{-1}u, g^{-1}u}) \right). \end{equation}
Then
\begin{eqnarray*} (g \cdot A)_{uu} &=& g \cdot (A_{g^{-1}u, g^{-1}u}) \\ &=&
g \cdot \left( 2 a \chi^A(g^{-1} u) g^{-1}u + h^A(f_{g^{-1}u, g^{-1}u}) \right)\\ & =&
2 a (g \cdot \chi^A) (u) \ u + g \left( h^A ( f_{g^{-1}u, g^{-1}u}) \right)  \\ &=& 2 a \chi^{g \cdot A}(u) \ u + h^{g \cdot A}(f_{uu})
,\end{eqnarray*}  where we applied  Lemma \ref{l.action} (2) and  (\ref{e2}) to derive the last line.
Thus $g \cdot A \in \Xi^a_Y.$ \end{proof}

\section{Proof of ${\rm Sec}({\rm Sing}(Y)) = Y$}\label{s.cubic1}

In this section, we will prove the following.

  \begin{theorem}\label{t.cubic1}
 Let $Y \subset \BP W$ be an irreducible cubic hypersurface with nonzero Hessian defined by a cubic form $f \in \Sym^3 W^*$.
 Assume that
there exists a nonzero  $ A \in \aut(\widehat{Y})^{(1)}$ of the form
    $$A_{uv} = a \ \chi^A (u) v + a \ \chi^A (v) u + h( f_{uv}) \mbox{ for } u, v \in W,$$
for some $h \in \Hom (W^*, W)$ and
$a \in \C, a \neq \frac{1}{4}.$ Then the secant variety of ${\rm Sing}(Y)$ coincides with the cubic hypersurface $Y$. \end{theorem}

To prove this, we assume ${\rm Sec}({\rm Sing}(Y)) \neq Y$ and derive a contradiction.
We start with the following geometric implication of this assumption.

\begin{proposition}\label{p.sB}
Let $Y \subset \BP W$ be an irreducible cubic hypersurface with nonzero Hessian defined by a cubic form $f \in \Sym^3 W^*$.   Assume that  $Y \neq {\rm Sec}({\rm Sing}(Y))$.
\begin{itemize} \item[(1)] For a general $w \in \widehat{Y}$, the Gauss subspace $\Gamma_w$ contains a linear subspace $B_w \subset \Gamma_w$ of codimension 1 such that  $\BP B_w = \BP \Gamma_w \cap {\rm Sing}(Y).$ \item[(2)] Let $\sB \subset W$ be the linear span of the union of $B_w$'s as $w$ varies over general points of $\widehat{Y}$. Then $\sB$ is preserved under $\aut(\widehat{Y})$, i.e. for any $\varphi \in \aut(\widehat{Y})$ and $v \in \sB$, the image $\varphi(v)$ is in $\sB$. \end{itemize} \end{proposition}

\begin{proof}
By Proposition \ref{p.Gauss} (2), the  fiber of the Gauss map of $\widehat{Y}$ through $w$ is $\Gamma_w$. By Proposition \ref{p.Gauss} (3),  the set-theoretic intersection ${\bf B}:= \BP \Gamma_w \cap {\rm Sing}(Y)$ is a hypersurface in $\BP \Gamma_w$. If ${\bf B}$ is a hypersurface of degree bigger than 1 in $\BP \Gamma_w$,  the secant variety of ${\bf B}$ must be the whole $\BP \Gamma_w$. As $[w]$ is a  general point of $Y$, this implies that the secant variety of ${\rm Sing}(Y)$  is $Y$, a contradiction to the assumption. It follows that ${\bf B}$ must be a linear subspace. Thus we have a linear subspace $B_w  \subset \Gamma_w$ of codimension 1 satisfying (1).

From the definition of $B_w$, any element $g \in {\rm Aut}_o(\widehat{Y})$ sends $B_w$ into $B_{gw}.$ Thus $g(\sB) = \sB$, which implies (2). \end{proof}

   \begin{proposition}\label{p.sigma}  Let $Y \subset \BP W$ be an irreducible cubic hypersurface with nonzero Hessian defined by a cubic form $f \in \Sym^3 W^*$.   Assume that  $Y \neq {\rm Sec}({\rm Sing}(Y))$ and we have a nonzero $A \in \aut(\widehat{Y})^{(1)}$. Let $\sH^A \subset W$ be the hyperplane defined by $\chi^A = 0.$ Define $\sigma^A \in \Hom (S^2 W, W)$ by $$\sigma^A_{uv} :=   A_{uv} - \frac{\chi^A(u)}{4} v - \frac{\chi^A(v)}{4} u \mbox{ for } u, v \in W.$$ For a general point  $w \in {\rm Sm}(\widehat{Y})$, let $B_w \subset \Gamma_w$ be as in Proposition \ref{p.sB}.
 Then  $\sigma^A_{ww} \in B_w$ and  $B_w = \Gamma_w \cap \sH^A.$  In particular, for the subspace $\sB \subset W$ defined in Proposition \ref{p.sB}, we have
 $\sigma^A \in \Hom(\Sym^2 W, \sB)$ and $\sB \subset \sH^A$.   \end{proposition}

\begin{proof}
Fix a general $w \in {\rm Sm}(\widehat{Y})$.
Recall that $A_{ww} \in \Gamma_w$ from Proposition \ref{p.Gamma}.
   From Proposition \ref{p.A_xx}, we know that $w$ and $A_{ww}$ are two linearly independent vectors in $\Gamma_w$. Thus the line $$\{w_t := tw + A_{ww}, t \in \C\} \subset \Gamma_w$$ will intersect $B_w$ at one point. The definition of $B_w$ in Proposition \ref{p.sB} implies that $w_t \in {\rm Sing}(\widehat{Y})$ for exactly one value of $t$. In other words, $$f(w_t, w_t, u) =0 \mbox{ for all } u \in W$$ has exactly one solution in $t$. But
   \begin{eqnarray*} f(w_t, w_t, u) &=& f(tw + A_{ww}, tw +  A_{ww}, u) \\
    &=& f(w, w, u) t^2 + 2 f(w, A_{ww}, u) t
   + f(A_{ww}, A_{ww}, u). \end{eqnarray*} The vanishing of  the discriminant of this quadratic equation in $t$ gives
   $$ f(w, A_{ww}, u)^2 = f(w,w,u) f(A_{ww}, A_{ww}, u) \mbox{ for } u \in W.$$    Using $f(w, w, u) \not\equiv 0$ and  Proposition \ref{p.Gamma}, we obtain $$
   f(A_{ww}, A_{ww}, u) = \left( \frac{\chi^A(w)}{2} \right)^2 f (w, w, u).$$
       The unique root  of the quadratic equation is
   $$ t = -  \frac{f(w, A_{ww}, u)}{f(w,w,u)} = - \frac{ \chi^A(w)}{2}$$ and for this value of $t$, we have $w_t = A_{ww} - \frac{\chi^A(w)}{2} w \in B_w.$ This shows that $\sigma^A_{ww} \in B_w$.

Note that  the endomorphism $A_w \in \aut(\widehat{Y}) \subset \End (W)$ preserves both ${\rm Sing}(\widehat{Y})$ and $\Gamma_w$ by Proposition \ref{p.varphi} (3). This implies that  $A_{ws} \in B_w$ for any  $s \in B_w.$  Then $$A_{ws} =   \frac{ \chi^A(w)}{4} s + \frac{\chi^A(s)}{4} w + \frac{\sigma_{ws}}{2}$$ implies
$\frac{\chi^A(s)}{4} w \in B_w.$ This shows $\chi^A(s) = 0$ because $w \not\in {\rm Sing}(\widehat{Y})$. Consequently, we have $B_w \subset \sH^A.$
\end{proof}

\begin{proposition}\label{p.check}
In the setting of Proposition \ref{p.sigma}, let $\Xi_o^a$ be the subset of $ \Xi^a_Y \subset
\aut(\widehat{Y})^{(1)}$ in Definition \ref{d.Xi}, consisting of elements $A \in \Xi^a_Y$ such that
the hyperplane $\BP \sH^A \subset \BP W$ of Proposition \ref{p.sigma} is sent to a hypersurface in $\BP W^*$ by the polar map $\Phi: \BP W \dasharrow \BP W^*$. In other words, an element $A \in \Xi^a_Y$ is in $\Xi_o^a$ if and only if the hyperplane
$\BP \sH^A$ is not contracted by $\Phi$. Assume that ${\rm Sec}({\rm Sing}(Y)) \neq Y$ and let $\sB$ be as in Proposition \ref{p.sB}.
If $a \neq \frac{1}{4}$ and  $A \in \Xi_o^a$, then $\sB = \sH^A$.
\end{proposition}

\begin{proof}
Fix $A \in \Xi_o^a$ with $a \neq \frac{1}{4}$. We have
$$A_{uu} = 2 a \chi^A(u) u + h^A(f_{uu}) \mbox{ for } u \in W.$$
We claim that $ {\rm Im}(h^A) + \sB = W$.
From Proposition \ref{p.sigma}, we have for $w  \in \widehat{Y}$
\begin{equation}\label{e.hA}  2a \chi^A (w)w + h^A (f_{ww}) = A_{ww} = \frac{\chi^A(w)}{2} w + \sigma^A_{ww}.\end{equation}
It follows that $-h^A(f_{ww}) + \sigma^A_{ww} =  2(a -\frac{1}{4}) \chi^A(w) w. $ Since $a \neq \frac{1}{4}$ and $\chi^A(w) \neq 0$ for a general $w \in \widehat{Y}$, this implies the claim.

For $w \in \sH^A$, we have from (\ref{e.hA})  $$ h^A(f_{ww}) = A_{ww} = \sigma^A_{ww} \in \sB.$$  This implies that $h^A(\widehat{\Phi(\BP \sH^A)}) \subset \BP \sB,$ because $h^A (\widehat{\Phi([w])})
= h^A (\widehat{f_{ww}})$ for $[w] \not\in {\rm Sing}(Y).$
 Since $A \in \Xi_o^a$, this shows that ${\rm Im}(h^A) \cap \sB$ has codimension at most 1 in ${\rm Im}(h^A)$. Combining this with the above claim  $ {\rm Im}(h^A) + \sB = W,$ we see that  the linear subspace $\sB$ has codimension at most one in $W$.
But $\sB \subset \sH^A$ by Proposition \ref{p.sigma}. Thus we have $\sB = \sH^A$. \end{proof}

\begin{proposition}\label{p.1dim}
For an irreducible cubic hypersurface $Y \subset \BP W$ with nonzero Hessian satisfying $Y \neq {\rm Sec}({\rm Sing}(Y))$ and $a \neq \frac{1}{4}$, either $\dim \Xi^a_Y \leq 1$ or
the closure of $\Xi_o^a$ in $\Xi^a_Y$ is a 1-dimensional subspace. \end{proposition}

\begin{proof}
Suppose that $\dim \Xi^a_Y \geq 2$.
By Proposition \ref{p.injective}, the family of hyperplanes $\sH^A, A \in \Xi^a_Y$ cover
$W$. Thus for some $A \in \Xi^a_Y$, the hyperplane $\BP \sH^A$ is not contracted by $\Phi$.
We conclude that $\Xi_o^a \neq \emptyset$.  By Proposition \ref{p.check},
all elements of $\Xi_o^a, a \neq \frac{1}{4}$, must be proportional to each other.
It follows that the closure of $\Xi_o^a$ is a 1-dimensional subspace. \end{proof}

\begin{proposition}\label{p.Y^+}
  In the setting of Proposition \ref{p.1dim}, assume that $\Xi^a_Y \neq 0$ for some $a \neq \frac{1}{4}$. Then there exists $A \in \Xi^a_Y$ such that the image
 $h^A(\widehat{Y}^*) \subset W$ of the affine cone of the dual variety $\widehat{Y}^* \subset W^*$   is invariant under the action of  ${\rm Aut}_o(\widehat{Y}) \subset {\bf GL}(W).$
 Consequently,  for any $\varphi \in \aut(\widehat{Y}) \subset \End(W)$ and a general point $w \in h^A(\widehat{Y}^*)$, the vector $\varphi(w)$ is contained in $T_w h^A(\widehat{Y}^*).$ \end{proposition}

  \begin{proof}
   If $\dim \Xi^a_Y =1$, choose any nonzero element of $\Xi^a_Y$ as $A$.   The natural action in Proposition \ref{p.invariance} of the Lie group ${\rm Aut}_o(\widehat{Y}) \subset {\bf GL}(W)$ on the 1-dimensional space $\Xi^a_Y$ sends $A$ to constant multiples of $A$. Thus the action of ${\rm Aut}_o(\widehat{Y})$ on $W$  preserves the set $h^A(\widehat{Y}^*)$.

   If $\dim \Xi^a_Y \geq 2$, then the closure of $\Xi^a_o$ is a 1-dimensional subspace in $\Xi^a_Y$
   by Proposition \ref{p.1dim}. Choose any nonzero element of $\Xi^a_o$ as $A$. It is clear that the action of ${\rm Aut}_o(\widehat{Y})$ on $\Xi^a_Y$ in Proposition \ref{p.invariance} preserves $\Xi^a_o$, because the polar map $\Phi: \BP W \dasharrow \BP W^*$ is equivariant under the natural actions of ${\rm Aut}_o(\widehat{Y})$. Thus ${\rm Aut}_o(\widehat{Y})$ sends $A$ to constant multiplies of $A$, which implies that it preserves $h^A(\widehat{Y}^*)$. \end{proof}

Now we are ready to prove Theorem \ref{t.cubic1}.

\begin{proof}[Proof of Theorem \ref{t.cubic1}] Assuming the contrary, we are in the situation of Proposition \ref{p.Y^+}. We can assume that $A$ is chosen as in Proposition \ref{p.Y^+}.

Fix a general $w \in {\rm Sm}(\widehat{Y})$ and write $s:= h^A(f_{ww}).$  Let $B_w \subset \Gamma_w$ and $\sigma^A$  be as in Proposition \ref{p.sB} and Proposition \ref{p.sigma}. We make the following three claims (C1) - (C3).

\begin{itemize}
\item[(C1)]  $s \in \Gamma_w \setminus B_w$.  \item[(C2)] $s \in {\rm Sm}(h^A(\widehat{Y}^*))$ and $\Gamma_w \cap T_s h^A(\widehat{Y}^*) = \widehat{s}.$
    \item[(C3)] $A_{su} = 0$ if $u \in B_w$.
   \end{itemize}

    Let us prove the claims.  Applying Proposition \ref{p.sigma}, we obtain
  $$2a \chi^A(w) w + h^A(f_{ww}) = A_{ww} = \frac{\chi^A(w)}{2} w + \sigma^A_{ww}.$$ Consequently,
    $$h^A(f_{ww}) + 2 (a-\frac{1}{4}) \chi^A(w) w = \sigma^A_{ww}  \in B_w.$$ Since $a \neq \frac{1}{4}$ and $w \not\in B_w \subset {\rm Sing}(\widehat{Y})$ from $w \in {\rm Sm}(\widehat{Y})$, we have  $s = h^A(f_{ww}) \not\in B_w$. On the other hand,
     Proposition \ref{p.Gamma} gives
    $$s= h^A(f_{ww}) = A_{ww} - 2a \chi^A(w) w \in \Gamma_w.$$ This proves (C1).

  Since $\Gamma_w$ is a general fiber of the Gauss map  $\Phi|_Y : Y \dasharrow Y^*,$ the claim (C1) implies that $h^A$ induces a rational map $\sigma: Y^* \dasharrow Y$ which gives a regular section of  the morphism  $$\Phi|_{{\rm Sm}(Y)} : {\rm Sm}(Y) \to Y^*$$
   over an open subset of $Y^*$. Thus by the generality of $w \in {\rm Sm}(\widehat{Y})$, we see that $[h^A(f_{ww})]$ is a nonsingular point of $
   \sigma(Y^*)$ where the fiber $\BP \Gamma_w$ and $\sigma(Y^*)$ intersect transversally.
   This proves (C2).

Pick $u \in B_w$ and let $\varphi= A_u \in {\rm aut}(\widehat{Y})$.  Then $\varphi(w) \in \Gamma_w$ and we have
$\varphi(s) \in \Gamma_w$ by Proposition \ref{p.varphi}.  On the one hand, by the choice of $A$ in Proposition \ref{p.Y^+},
we have $\varphi(s) \in T_s h^A(\widehat{Y}^*).$ Thus by the claim (C2), we have $A_{us} \in \widehat{s}$. On the other hand, the endomorphism $A_s$ must be tangent to ${\rm Sing}(\widehat{Y})$. Thus $A_{su} \in B_u$.
Since $\widehat{s} \cap B_w =0$, we conclude that $A_{us} =0$, proving (C3).

   \medskip
   Now, for a general point $v \in \Gamma_w$, we have $f_{vv} \in \C \cdot f_{ww}$ and
   $h^A(f_{vv}) = q_{vv} s$ for some $q_{vv} \in \C$. By continuity, this must hold for all $v \in \Gamma_w$. Consequently, we have a quadratic form $q$ on $\Gamma_w$ such that
   $$ A_{vv} = 2a \chi^A(v) v + q_{vv} s\mbox{ for } v \in \Gamma_w$$ and by polarizing it, \begin{equation}\label{e.q} A_{uv} = a \chi^A(u) v + a \chi^A(v) u + q_{uv} s\mbox{ for } u, v \in \Gamma_w \end{equation}

  Choosing $0 \neq u \in B_w$  and $ v= s$ in (\ref{e.q}), the claim (C3) gives $$0= A_{us} = a \chi^A(s) u + q_{us} s.$$ Thus $$0 \neq a \chi^A(s) u = - q_{us} s.$$ Since the righthand side cannot be in $B_w$ unless $q_{us} =0$ by the claim (C1), we have a contradiction. \end{proof}

\section{Proof of Theorem \ref{t.cubic} from ${\rm Sec}({\rm Sing}(Y)) = Y$}\label{s.cubic2}

  Theorem \ref{t.cubic} follows from Theorem \ref{t.cubic1} and the next theorem.

\begin{theorem}\label{t.cubic2}
 Let $Y \subset \BP W$ be an irreducible cubic  with nonzero Hessian defined by a cubic form $f \in \Sym^3 W^*$.
 Assume that \begin{itemize} \item[(a)] the singular locus
${\rm Sing}(Y)$ is nonsingular; and \item[(b)]
there exists a nonzero  $ A \in \aut(\widehat{Y})^{(1)}$ of the form
    $$A_{uv} = a \ \chi^A (u) v + a \ \chi^A (u) v + h( f_{uv}) \mbox{ for } u, v \in W,$$
for some $h \in \Hom (W^*, W)$ and
$a \in \C, a \neq \frac{1}{4}.$ \end{itemize}
  If $Y= {\rm Sec}({\rm Sing}(Y))$, then there exists an irreducible component $S$ of ${\rm Sing}(Y)$ such that $Y= {\rm Sec}(S)$. \end{theorem}

 In fact, in the setting of Theorem \ref{t.cubic}, Theorem \ref{t.cubic1} implies that we may assume the condition $Y= {\rm Sec}({\rm Sing}(Y))$ required for Theorem \ref{t.cubic2}. Then $S$ in Theorem \ref{t.cubic2} is an irreducible nondegenerate nonsingular  variety with $\aut(\widehat{S})^{(1)}
  = \aut(\widehat{Y})^{(1)} \neq 0$. Thus by Theorem \ref{t.Severi}, we see that $S$ is a Severi variety. This proves Theorem \ref{t.cubic}.

For the proof of Theorem \ref{t.cubic2}, we use
the following two elementary lemmata from linear algebra.

\begin{lemma}\label{l.UV}
Let $\sH \subset W$ be a hyperplane in a vector space $W$ defined by a nonzero linear functional $\lambda \in W^*$. Suppose that we have two subspaces $U, V \subset W$ such that
$$U \not\subset \sH, \ V \not\subset \sH \mbox{ and } W = U + V.$$
Then the linear span of the set $$\{ \lambda (v) \ u + \lambda (u) \ v, \ u \in U, v \in V\}$$
has codimension at most 1 in  $W$. \end{lemma}

\begin{proof}
  Suppose $\eta \in W^*$ is  a linear functional annihilating the linear span. Then $$\lambda(v) \ \eta(u ) + \lambda (u) \ \eta(v) = 0 \mbox { for any } u \in U, \ v \in V.$$ So we have a constant $c \in \C$ satisfying
 $$c = \frac{\eta(u)}{\lambda(u)}  = - \frac{\eta (v)}{\lambda (v)}    \mbox{ for any } u \in U \setminus \sH, \ v \in V \setminus \sH. $$
 Then $$\eta(u) = c \lambda (u) \mbox{ for all } u \in U \mbox{ and }
 \eta(v) = -c \lambda (v) \mbox{ for all } v \in V.$$
Thus the linear functional $\eta$ is uniquely determined up to a scalar multiple. \end{proof}

\begin{lemma}\label{l.Aphi}
For a vector space $W$ and the hyperplane $\sH \subset W$ defined by a nonzero linear functional $\lambda \in W^*$, suppose
 $A \in \Hom (S^2 W, W)$  satisfies $A_{uv} =0$ if $u, v \in \sH$. Fix a vector $p \in W$ satisfying $\lambda(p) = 1$. Then there exists $\varphi \in \End(W)$ such that $\varphi(p) =0$ and $$A_{uv} = \lambda(u) \varphi(v) + \lambda(v) \varphi(u) + \lambda(u) \lambda(v) A_{pp} \mbox{ for any } u, v \in W.$$ \end{lemma}

\begin{proof}
 For each element  $u \in \sH$, the endomorphism $A_{u}: w \mapsto A_{wu}$ of $W$ annihilates $\sH$. Thus there exists a unique vector $\varphi(u) \in W$ satisfying
$A_{wu} = \lambda(w) \varphi(u)$ for any $w \in W$. This defines $\varphi \in \Hom (\sH, W)$. Extend it to an element  $\varphi \in \End(W)$ by putting $\varphi (p) =0$.
Define $A' \in \Hom(S^2 W,W)$ by $$A'_{uv} = \lambda(v) \varphi(u) + \lambda(u) \varphi(v) \mbox{ for any } u, v \in W.$$
Then $B:=A - A'$ satisfies $B_{uv} = 0$ if $u \in \sH$ or $v \in \sH$.
Thus  $B_{uv} = \lambda(u) \lambda(v) q$ for some $q \in W$. Since
$A'_{pp} =0$, we see that $q = A_{pp}$.
\end{proof}

\begin{proof}[Proof of Theorem \ref{t.cubic2}]
 To prove the theorem, assume that  $Y \neq {\rm Sec}(Z)$ for any irreducible component $Z$ of ${\rm Sing}(Y)$. We will derive a contradiction.

To start with, we derive a number of geometric consequences of this assumption in the following lemma.
 Denote  by $\langle Z \rangle \subset W$ the linear subspace spanned by a variety $Z \subset \BP W$.

  \begin{lemma}\label{L.1}
  There are two distinct irreducible components $S, S'$ of ${\rm Sing}(Y)$ with the following properties.
\begin{itemize}
\item[(1)] $Y = {\rm Join}(S, S') := \cup_{x \in S, y \in S'} \overline{xy}.$
\item[(2)] $ \widehat{S} \neq \langle S \rangle, \ \widehat{S'} \neq \langle S' \rangle \mbox{ and } W = \langle S \rangle + \langle S' \rangle$.
 \item[(3)] $\widehat{S} \cap \langle S' \rangle \neq 0 \neq \widehat{S'} \cap \langle S \rangle$.
     \item[(4)] Let $F \subset Y$ be a linear subspace contained in $Y$ such that
     $F \cap S \neq \emptyset \neq F \cap S'$. Suppose there exits a point $z \in F \cap {\rm Sm}(Y)$ such that $F$ is contained in the fiber through $z$  of the Gauss map $\gamma_Y.$ Then $F$ is a line $\overline{xy}$ with $x \in S, y \in S'$.
\end{itemize} \end{lemma}

\begin{proof}
By the requirement that $Y$ is not the secant variety of a single irreducible component of ${\rm Sing}(Y)$,  there are two distinct irreducible components $S$ and $S'$ of ${\rm Sing}(Y)$
such that $Y = {\rm Sec}(S \cup S').$ The irreducibility of $Y$ implies that $Y$ is equal to ${\rm Join}(S, S')$, proving (1).

Since $Y$ is not a cone, (1) implies (2).

Since $\widehat{S} \neq \langle S \rangle$, there is a point $$z \in {\rm Sec}(S)
\setminus S  \ \subset \ Y.$$ Since $z \in {\rm Join}(S, S')$ by (1), we have $x \in S, y \in S'$ such that $z \in \overline{xy}$.
If $\widehat{y} \not\subset \langle S \rangle$, then $z \in \overline{xy} \cap \BP \langle S \rangle = \{x\},$ which is absurd.
Thus $0 \neq \widehat{y} \subset \langle S \rangle \cap \widehat{S}'$. Similarly, we have  $\widehat{S} \cap \langle S' \rangle  \neq 0,$ proving (3).

To prove (4),
 note that the intersection $F \cap {\rm Sing}(Y)$ is a hypersurface in $F$ by Proposition \ref{p.Gauss} (3).  So the intersections $F \cap S$ and $F\cap S'$ are two disjoint hypersurfaces in the linear space $F$. This implies that $\dim F =1$ and $F= \overline{xy}$
 for some $x \in S$ and $y\in S'$.
 \end{proof}

Lemma \ref{L.1} has the following implication for $A$.

\begin{lemma}\label{L.2}
\begin{itemize} \item[(1)] $A_{uv} = 0$ for $u \in \langle S\rangle$ and $v \in \langle S' \rangle.$ \item[(2)] $a \chi^A \neq 0.$
\end{itemize} \end{lemma}

\begin{proof}
By  Lemma \ref{L.1} (2),  we can pick $u \in \widehat{S}$ and $v \in \widehat{S}'$ generally such that
   \begin{equation}\label{e.SS} T_u\widehat{S} \cap \langle u, v \rangle = \widehat{u}  \ \mbox{ and } \
T_v \widehat{S'} \cap \langle u, v \rangle = \widehat{v}. \end{equation}
By Proposition \ref{p.Gamma}, Proposition \ref{p.Gauss} (2) and Lemma \ref{L.1} (4), we have  $A_{ww} \in \langle u, v \rangle$ for any general $w \in \langle u, v \rangle$.   Thus  $$A_{st} \in \langle u, v \rangle \mbox{ for any } s, t \in \langle u, v \rangle. $$
In particular, we have $A_{uv} \in \langle u, v \rangle$.
Since $S$ and $S'$ are preserved by $\aut(\widehat{Y})$,  we have $A_{tu} \in T_u \widehat{S} $ and $A_{tv}
\in T_v \widehat{S'}$ for any $t \in W$.  Thus $$A_{uv} \in T_u \widehat{S} \cap T_v \widehat{S'} \cap \langle u,v \rangle.$$
Then (\ref{e.SS}) implies
 $A_{uv} \in \widehat{u} \cap \widehat{v} =0.$
Since this is true for general $u \in \widehat{S}$ and $v \in \widehat{S}'$, we have $A_{uv} = 0$ for all $u \in \langle S \rangle$ and $v \in \langle S' \rangle.$ This proves (1).

If $a \chi^A = 0$, then for $u, v \in {\rm Sing}(\widehat{Y})$, we have
$A_{uu} = 0 = A_{vv}$ because $f_{uu} = f_{vv} = 0$. Then for a general $w \in \widehat{Y}$ with $w = u + v, u \in \widehat{S}, v \in \widehat{S}',$ (1) implies that $$A_{ww}
 = A_{u+v, u+v} = A_{uu} + 2 A_{uv} + A_{vv} = 0.$$ Thus $A_{ww} = 0$ for all $ w \in \widehat{Y}$. Since $Y$ is an irreducible cubic hypersurface, this implies $A_{ww} = 0$ for all $w \in W$, contradicting $A\neq 0$. This proves (2).\end{proof}

\begin{lemma}\label{L.3}
 Let $R = \langle S \rangle \ \cap \langle S' \rangle.$
 Then \begin{itemize} \item[(1)] $R \subset \widehat{Y}$; \item[(2)] $R \not\subset {\rm Sing}(\widehat{Y})$; and  \item[(3)] $h^A(f_{rr}) = 0$ for any $r \in R$ and ${\rm Ker}(h^A) \neq 0$.
    \end{itemize} \end{lemma}

\begin{proof}
For $r \in R$ and a general $w \in W,$ Lemma \ref{l.Aut} (2)  gives
 $$ f(A_{wr}, r, r) + f(r, A_{wr}, r) + f(r, r, A_{wr}) = \chi^A(w) \ f(r,r,r).$$
 Since the left hand side vanishes by Lemma \ref{L.2} (1), we obtain $r \in \widehat{Y}$.
 Thus $R \subset \widehat{Y}$, proving (1).

 If $R \subset {\rm Sing}(\widehat{Y})$,
 let $S^{''}$ be the irreducible component of ${\rm Sing}(Y)$ containing $\BP R$.
 By Lemma \ref{L.1} (3), we have
 $$ \widehat{S} \cap R = \widehat{S} \cap \langle S \rangle \cap \langle S' \rangle =
 \widehat{S} \cap \langle S' \rangle \neq 0.$$ Thus
 $S \cap S^{''} \neq \emptyset$ and similarly,   $S' \cap S^{''} \neq \emptyset.$
This contradicts the assumption that ${\rm Sing}(Y)$ is nonsingular. This proves (2).

For any $r \in R$ and a general $w \in W,$ Lemma \ref{l.Aut} (2)  gives
 $$ f(A_{rw}, w, w) + f(w, A_{rw}, w) + f(w, w, A_{rw}) = \chi^A(r) \ f(w,w,w).$$
 Since the left hand side vanishes by Lemma \ref{L.2} (1), we have $\chi^A(r) = 0.$
 Apply Lemma \ref{L.2} (1) to have $$0 = A_{rr} = 2 a \chi^A(r) r + h^A (f_{rr})  = h^A(f_{rr}).$$ Thus $f_{rr} \in {\rm Ker}(h^A).$ By (2), there exists an element $r \in R \cap {\rm Sm}(\widehat{Y})$, i.e. $f_{rr} \neq 0$. Thus ${\rm Ker}(h^A) \neq 0$.
\end{proof}

\begin{lemma}\label{L.4} Either $\chi^A(\widehat{S}) = 0$ or $\chi^A(\widehat{S}') = 0$.  \end{lemma}

    \begin{proof}
    Assume that $\chi^A(\widehat{S}) \not\equiv 0$ and $\chi^A(\widehat{S}') \not\equiv 0$.
For any $u \in \langle S \rangle$ and $v \in \langle S' \rangle$, Lemma \ref{L.2}(1) gives $$0 = A_{uv} = a \chi^A(u) v + a\chi^A(v) u + h^A (f_{uv}).$$ Thus $h^A(f_{uv}) = - a \chi^A (u) v- a \chi^A(v) u$. By Lemma \ref{l.UV} with $U = \langle S \rangle, V = \langle S' \rangle$ and
$\lambda = a \chi^A \neq 0$ (from Lemma \ref{L.2} (2)), we see that  the values of $h^A(f_{uv})$ span a subspace of codimension at most 1 in $W$. Since  $h^A$ is not surjective by
Lemma \ref{L.3} (3), we see that ${\rm Im}(h^A)$ is a hyperplane in $W$. Thus ${\rm Ker}(h^A) \subset W^*$ is 1-dimensional. Then Lemma \ref{L.3} (3) implies
that  the  Gauss map $\gamma_{Y}$ sends $\BP R$ to the point $\BP {\rm Ker}(h^A)$. Thus by Lemma \ref{L.1} (4),  we conclude that $\BP R$ is a line in $\BP W$ intersecting $S'$ at one point $s' \in S'$, which implies
$$S' \cap \BP \langle S \rangle = \{s' \}.$$
 Then from Lemma \ref{L.1} (1),  $${\rm Sec}(S) \subset Y \cap \BP \langle S \rangle = {\rm Join}(S,S') \cap \BP \langle S \rangle =  {\rm Join}(s', S).$$
  This shows that $S$ must be a hypersurface in $ \BP \langle S \rangle$, which is not a hyperplane because of Lemma \ref{L.1} (2).  Fix a general $w \in W$ such that
 $f_w \in \Sym^2 W^*$   induces a nonzero quadratic equation $q:=f_w|_{\langle S \rangle}.$ Since the hypersurface $S \subset \BP \langle S \rangle$ satisfies this quadratic equation $q=0$, no point of $\BP \langle S \rangle$ outside $S$ satisfies this equation.    But  $s' \in {\rm Sing}(Y)$ satisfies the equation, too. A contradiction to $s' \not\in S$. \end{proof}

Now to finish the proof of Theorem \ref{t.cubic2},
 we will assume that $\chi^A (\widehat{S}') =0$ from
 Lemma \ref{L.4}.
 Let $\sH \subset W$ be the hyperplane defined by $\chi^A = 0$ such that $\widehat{S}' \subset \sH$.  We will consider
two cases separately: when $A|_{\sH} = 0$ (meaning $A_{uv} = 0$ for any $u, v \in \sH$) and when $A|_{\sH} \neq 0.$

{\bf  (Case 1)} When $A|_{\sH} = 0$. Applying Lemma \ref{l.Aphi} with $\lambda= a \chi^A$, we have $\varphi \in \End(W)$ and $p \in W \setminus \sH$ such that  \begin{equation}\label{e.star} A_{uv} = \lambda(u) \ \varphi(v) + \lambda(v) \ \varphi(u) + \lambda (u) \lambda(v) A_{pp} \mbox{ for } u, v \in W.\end{equation}
Choosing  $u \in \widehat{S} \setminus \sH$ and $v \in \widehat{S}' \subset \sH$, Lemma \ref{L.2} gives
$$0 = A_{uv} = \lambda(u) \varphi(v).$$
It follows that $\langle S' \rangle \subset {\rm Ker}(\varphi).$
By Lemma \ref{L.1} (3), we can choose an arc $\{ u_t \in \widehat{S}, t \in \Delta\}$ such that $0 \neq u_0 \in \langle S' \rangle$ and $\lambda(u_t) \neq 0$ for $t \neq 0$. Since $f_{u_t u_t} =0$ from $u_t \in {\rm Sing}(\widehat{Y})$, we have
\begin{eqnarray*} A_{u_t u_t} & = & 2a \chi^A(u_t) u_t + h^A(f_{u_t u_t}) \\
& = & 2 \lambda(u_t) u_t. \end{eqnarray*} On the other hand, the equation (\ref{e.star}) gives
$$A_{u_t u_t} = 2 \lambda(u_t) \ \varphi(u_t) + \lambda(u_t)^2 \ A_{pp}.$$
From the two equations, we obtain
 $u_t = \varphi( u_t) + \frac{1}{2} \lambda(u_t) A_{pp}$. But at $t=0$, this gives $u_0 =0$ from
$u_0 \in \langle S' \rangle \subset {\rm Ker} (\varphi) \cap \sH$. This contradicts our choice
$u_0 \neq 0$.

{\bf (Case 2)} When  $A|_{\sH} \neq 0$.  The set
 $$J : = \{w \in \sH, \ A_{ww} = 0\}$$ is strictly smaller than $\sH$.   We claim that $\widehat{Y} \cap \sH$ is contained in $J$. Pick a general $w \in \widehat{Y} \cap \sH.$
 To prove the claim, it suffices to show $A_{ww} = 0.$
Using the notation $\lambda= a \chi^A$  and writing  $w \in \sH \cap \widehat{Y}$ as $w = u + v, u \in \widehat{S}, v \in \widehat{S}'$ from Lemma \ref{L.1} (1),  we have $\lambda(v) = 0$ from $\widehat{S}' \subset \sH$.
But then $\lambda(w) = \lambda(u) = 0$ because $ w \in \sH$.
Consequently,  $$A_{uu} = 2 \lambda(u) u + h^A(f_{uu}) = 0 \mbox{ and } A_{vv} = 2 \lambda(v) v + h^A(f_{vv}) = 0.$$ Since we have $A_{uv} =0$ from Lemma \ref{L.2},
we have  $$A_{ww} = A_{uu} + 2 A_{uv} + A_{vv} =0.$$ This proves the claim.

 The intersection $\BP \sH \cap Y$ is defined by the cubic polynomial $f|_{\sH}$, while
  by the above claim it satisfies a nontrivial system of quadratic equations defining $J \subset \BP \sH$. Thus
we can write the divisor on $\sH$ defined by $f|_{\sH}$ as $2 H_1 \cup H_2$, a double hyperplane and another hyperplane in $\sH$.

Assume that $H_1 \neq H_2$. Any point of  $H_2 \setminus H_1$ is a nonsingular point of the scheme-theoretic intersection $\sH \cap (f=0),$ hence is a nonsingular point of $\widehat{Y}$. It follows that both $\widehat{S} \cap \sH$ and $ \widehat{S}'$ are contained in the subspace $ H_1 \subset \sH$. But $\BP H_2 \subset Y = {\rm Join}(S, S').$
Thus any point $ z \in\BP  H_2 \setminus \BP H_1$ is contained in  a line $\ell$ joining $S$ and $S'$.
     If $\ell \not\subset \BP \sH$, then  $\ell \cap \BP \sH \in S'$ which cannot contain $z$. If $\ell \subset \BP \sH$, it intersects $\BP H_1$ at two distinct points,
      one in $S \cap \BP \sH \subset \BP H_1$ and another in $S' \subset \BP H_1$. Thus it is
      contained in $\BP H_1$ and cannot contain $z$, either. So we have a contradiction.

   The only possibility is that $H_1 = H_2$ and the scheme-theoretic intersection is $\sH \cap (f=0) = 3 H_1$, a triple hyperplane in $\sH$.
   We can choose coordinates $(w_1, \ldots, w_n), n= \dim W,$ on $W$ such that $\sH = (w_n=0)$ and
   $H_1 = (w_n = w_{n-1} =0)$. From $\sH \cap (f=0) = 3 H_1$, a defining equation of $Y$
   can be written as
   $$f(w,w,w) =  w_{n-1}^3 + \sum_{i,j=1}^{n} q_{ij} w_i w_{j} w_n = 0$$
   for some $ q_{ij}=q_{ji} \in \C$. From
   $${\rm d} f = 3 w_{n-1}^2 {\rm d} w_{n-1} + \sum_{i,j=1}^{n} q_{ij} w_i w_{j} {\rm d} w_n
   + 2 w_n \sum_{i,j=1}^n q_{ij} w_j {\rm d} w_i,$$
   the variety ${\rm Sing}(\widehat{Y}) \cap \sH \cap H_1$ is the quadric hypersurface
   $$\sum_{i,j = 1}^{n-2} q_{ij} w_i w_j =0$$ in the linear space $w_{n-1} = w_n = 0.$
   If $n \geq 5,$ distinct components of this quadric hypersurface must intersect. This is a contradiction  because ${\rm Sing}(\widehat{Y}) \cap \sH \cap H_1$ should have at least two disjoint components, including $S'$ and components of $S \cap \langle S' \rangle$. On the other hand, if $\dim W \leq 4,$ it is easy to derive a contradiction from Lemma \ref{L.1}. This finishes the proof of Theorem \ref{t.cubic2}  \end{proof}

\medskip
{\bf Acknowledgment} We would like to thank Igor Dolgachev for  enjoyable discussions on cubic hypersurfaces. I am very grateful to Baohua Fu for a careful reading and detailed comments on the first draft of the paper.

   \end{document}